\documentclass[11pt]{article}

\usepackage{amsmath}
\usepackage{amssymb}
\usepackage{amsthm}
\usepackage{graphicx}
\usepackage{amscd}
\usepackage{epic, eepic}
\usepackage{url}
\usepackage{color}
\usepackage{todonotes}
\usepackage{comment}
\usepackage{hyperref}

\newtheorem{theorem}{\bf Theorem}[section]
\newtheorem{lemma}[theorem]{\bf Lemma}
\newtheorem{cor}[theorem]{\bf Corollary}
\newtheorem{proposition}[theorem]{\bf Proposition}

\newtheorem{prop}[theorem]{\bf Proposition}

\newtheorem{remark}[theorem]{\bf Remark}

\newtheorem{notation}[theorem]{\bf Notation}
\newtheorem{result}[theorem]{\bf Result}

\theoremstyle{definition} 
\newtheorem{defi}[theorem]{\bf Definition}

\newcommand{\cP}{\mathcal{P}}
\newcommand{\C}{\mathcal{C}}
\newcommand{\cC}{\mathcal{C}}
\newcommand{\cL}{\mathcal{L}}

\newcommand{\B}{\mathcal{B}}
\newcommand{\cB}{\mathcal{B}}
\newcommand{\PP}{\mathbb{P}}

\newcommand{\cS}{\mathcal{S}}

\newcommand{\F}{\mathbb{F}}
\newcommand{\E}{\mathbb{E}}

\newcommand{\PG}{\mathrm{PG}}
\newcommand{\AG}{\mathrm{AG}}
\newcommand{\GF}{\mathrm{GF}}

\newcommand{\cut}[1]{}

\newcommand{\angbra}[1]{\left\langle #1 \right\rangle}

\newcommand{\qbinom}[2]{\begin{bmatrix}#1\\ #2\end{bmatrix}_q}

\makeatletter
\@addtoreset{equation}{section}
\makeatother

\makeatletter
\@addtoreset{table}{section}
\makeatother

\title{Short minimal codes and covering codes via strong blocking sets in projective spaces}

\author{{\bf Tam\'as H\'eger\thanks{ELKH--ELTE Geometric and Algebraic Combinatorics Research Group,
			E\"otv\"os Lor\'and University, Budapest, Hungary. The author is supported by the Hungarian Research Grant (NKFI) No.  124950. 	E-mail: {\tt heger.tamas@ttk.elte.hu}}, } 
	{\bf Zolt\'an L\'or\'ant Nagy\thanks{ELKH--ELTE Geometric and Algebraic Combinatorics Research Group,
			E\"otv\"os Lor\'and University, Budapest, Hungary. The author is supported by the Hungarian Research Grant (NKFI) No. K 120154, 124950, 134953. 	E-mail: {\tt nagyzoli@caesar.elte.hu}}}\\ 
}
\date{}

\begin{document}
	
	\maketitle
	
	\begin{abstract} 
		Minimal linear codes   are in one-to-one correspondence with special types of blocking sets of projective spaces over a finite field, which are called strong or cutting blocking sets. In this paper we prove an upper bound on the minimal length of minimal codes of dimension $k$ over the $q$-element Galois field which is linear in both $q$ and $k$, hence improve the previous superlinear bounds. This result determines the minimal length up to a small constant factor.
		We also improve the lower and upper bounds on the size of so called higgledy-piggledy line sets in projective spaces and apply these results to present improved bounds on the size of covering codes and saturating sets in projective spaces as well. The contributions rely on geometric and probabilistic arguments.
		
	\end{abstract}

	\section{Introduction}

	Throughout this paper, $q$ denotes a prime power and $\F_{q}$ denotes the Galois field with $q$ elements, while $p$ stands for the characteristics of $\F_q$. 
	Let $\F_{q}^{n}$ be the $n$ -dimensional vector space over $\F_{q}$.
	Denote by $[n, r]_{q}$ a $q$-ary linear code of length $n$ and
	dimension $r$, which is the set of codewords (code vectors) of a subspace of $\F_{q}^{n}$ of dimension $r$. For a general introduction on codes we refer to \cite{huffman2010fundamentals}.
	
	\begin{defi}\label{def:mincode}
		In a linear code, a codeword is  
		{\em minimal} if its support does not contain the support of any codeword other than its scalar multiples. A code is {\em minimal} if  its codewords are all minimal.
	\end{defi}
	Minimal codewords in linear codes were originally
	studied in connection with decoding algorithms~\cite{MR551274} and have been used by Massey~\cite{Massey} to determine the access structure in his code-based secret sharing scheme. For a general overview on recent results in connection with minimal codes we refer to \cite{alfa+neri, lu2019parameters}. The general problem is to determine the minimal length of a  $[n,k]_q$ minimal code can have, provided that $k$ and $q$ are fixed.
	
	\begin{defi}[Minimal length of a minimal code]\label{min_length}
		Denote by $m(k,q)$ the minimal length of a  $[n,k]_q$ minimal code with parameters $k$ and $q$.
	\end{defi}
	
	The following bounds are due to Alfarano, Borello, Neri and Ravagnani \cite{alfa+neri}. We do not state the quadratic upper bound precisely, as it depends on some properties of $q$ and $k$.
	
	\begin{theorem}[\cite{alfa+neri}]\label{thm:lower_bound_length}
		Let $\C$ be an $[n,k]_q$ minimal code. We have 
		\begin{equation}\label{eq_1}
			(k-1)(q+1)\leq  m(k,q) \leq ck^2q,
		\end{equation}
		for some $c\geq 2/9$.
	\end{theorem} 
	
	In \cite{chabanne2013towards} it has been shown by Chabanne, Cohen and Patey that the upper bound on $m(k,q)$ can be refined as follows.

	\begin{theorem}[\cite{chabanne2013towards}]\label{Cohen}
		\begin{equation}\label{eq:upper_bound_nonconstructive} m(k,q) \leq \frac{2k}{\log_q\left(\frac{q^2}{q^2-q+1}\right)}.
		\end{equation}
	\end{theorem}
	
	Note that this bound is a non-constructive one and that $\lim_{q \rightarrow \infty} q\ln{q} \cdot \log_q\left(\frac{q^2}{q^2-q+1}\right) =1$, hence for large $q$, it roughly says $m(k,q)\lesssim 2kq\ln(q)$. For $q=2$ it yields $m(k,2)\leq2k/(\log_2(4/3))$.

	Our contribution is a linear upper bound in both $k$ and $q$. The two cases follow from Theorems \ref{main1} and \ref{thm:binary}.
	
	\begin{theorem}\label{probab}
		If $q>2$, then
		\begin{equation}\label{eq_2kq} m(k,q) \leq \left\lceil \frac{2}{1+\frac{1}{(q+1)^2\ln q}}(k-1)\right\rceil(q+1).
		\end{equation}
		If $q=2$, then
		\begin{equation} m(k,2)\leq \frac{2k-1}{\log_2(\frac{4}{3})}.\end{equation}
	\end{theorem}
	
	The case $q>2$ will follow immediately from Theorem \ref{main1}, which is based on a random construction of taking the point set of the union of less than $2k$ lines in a suitable projective space. From the proof it follows easily that with a positive probability (calculated therein), this provides a desired minimal code.

	Note that to apply probabilistic arguments for problems in finite geometry is not at all new. Here we only mention the paper of G\'acs and Sz\H{o}nyi \cite{Gacs} on various applications, the celebrated paper on complete arcs of Kim and Vu \cite{Kim-Vu} which applies  R\"odl's nibble, and the paper of second author \cite{NagyZ} on the  topic of saturating sets of projective planes which we revisit later on.
	
	As it was noticed  recently by  Alfarano, Borello and Neri \cite{alfarano2019geometric} and independently by Tang, Qiu, Liao, and Zhou \cite{tang2019full}, minimal codes are in one-to-one correspondence with special types of blocking sets of projective spaces, which they called {\em cutting blocking sets} after the earlier paper of Bonini and Borello \cite{bonini2020minimal}. In fact, this concept has been investigated in connection with saturating sets and covering codes a decade earlier by Davydov, Giulietti, Marcugini and Pambianco \cite{Davydov09} under the name \emph{strong blocking sets} and in the paper of Fancsali and Sziklai \cite{FancsaliSziklai1} in connection with so-called higgledy-piggledy line arrangements under the name {\em generator set}.
	
	Blocking sets and their generalisations are well-known concepts in finite geometry. For an introduction to finite geometries, blocking sets and various related topics we refer to \cite{Hirschfeld, KissSzonyi}. Let us give the corresponding definitions. We denote the finite projective geometry of dimension $N$ and order $q$ by $\PG(N,q)$.

	\begin{defi}[Blocking sets] Let $t, r, N$ be positive integers with $r<N$. A $t$-\em{fold} $r$-\em{blocking set} in~$\PG(N,q)$ is a set $\B\subseteq \PG(N,q)$ such that for every $(N-r)$-dimensional subspace $\Lambda$ of $\PG(N,q)$ we have $|\Lambda \cap \B|\geq t$. When $r=1$, we will  refer to $\B$ as a $t$-\em{fold blocking set}. When $t=1$, we will refer to it as an $r$-\em{blocking set}. When  $r=t=1$, $\B$ is simply a \em{blocking set}.
	\end{defi}

	A particular type of blocking sets was introduced independently in \cite{Davydov09} and \cite{bonini2020minimal} under different names. Let us present and use both terminologies.
	
	\begin{defi}[Multifold strong blocking sets, aka cutting blocking sets]\label{strong_def} 	 
		A {\em $t$-fold strong blocking set} of $\PG(N,q)$ is a point set that meets each $(t-1)$-dimensional subspace $\Lambda$ in a set of points which spans the whole subspace $\Lambda$ \cite{Davydov09}. A {\em cutting $t$-blocking set} of $\PG(N,q)$ is a point set that meets each $(N-t)$-dimensional subspace $\Lambda$ in a set of points which spans the whole subspace $\Lambda$ \cite{bonini2020minimal}. A \em{cutting blocking set} (without prefix) is a cutting $1$-blocking set.
	\end{defi}
	
	Clearly, cutting $t$-blocking sets and $(N-t+1)$-fold strong blocking sets coincide. It is worth noting that as we need at least $r+1$ points to span a subspace of dimension $r$, every $(r+1)$-fold strong blocking set (or, in other words, a cutting $(N-r)$-blocking set) is an $(r+1)$-fold $(N-r)$-blocking set.
	
	A cutting blocking set of $\PG(N,q)$ of size $n$ corresponds to a minimal $[n,N+1]_q$  code (see \cite{alfarano2019geometric, tang2019full}, and also Section \ref{sec:geom}). Thus, as short minimal codes are of interest, constructing small cutting blocking sets of $\PG(N,q)$ is highly relevant.
	
	Cutting blocking sets were also investigated by H\'eger, Patk\'os and Tak\'ats \cite{Heger} under the name hyperplane generating set. There it was proposed to construct such a set as the union of lines, and appropriate sets of lines were called line sets in {\em higgledy-piggledy arrangement}. 
	
	\begin{defi}
		A set of lines of $\PG(N,q)$ is in \emph{higgledy-piggledy arrangement}, if the union of their point sets is a cutting blocking set of $\PG(N,q)$. We may also refer to such line sets as higgledy-piggledy line sets for short.
	\end{defi}
	
	Such line sets of not necessarily finite projective spaces were studied in detail in \cite{FancsaliSziklai1} (see also \cite{FancsaliSziklai2} for a generalisation to higgledy-piggledy subspaces). Let us recall the three main results of Fancsali and Sziklai \cite{FancsaliSziklai1}.
	
	\begin{theorem}[Fancsali, Sziklai, \cite{FancsaliSziklai1}, Theorems 14, 24 and 26] \label{HPlines} Let $\mathbb{F}$ be an arbitrary field.
		\begin{enumerate}
			\item[i)] If $|\F|\geq N + \left\lfloor N/2\right\rfloor$, then every line set of $\PG(N,\F)$ in higgledy-piggledy arrangement contains at least $N + \left\lfloor N/2\right\rfloor$ lines.
			\item[ii)] If $|\F|\geq 2N - 1$, then there exist a line set of $\PG(N,\F)$ in higgledy-piggledy arrangement containing $2N-1$ lines.
			\item[iii)] If $\F$ is algebraically closed, then every line set of $\PG(N,\F)$ in higgledy-piggledy arrangement contains at least $2N - 1$ lines.
		\end{enumerate}
	\end{theorem}
	
	Note that for $2\leq N\leq 5$, there are line sets in higgledy-piggledy arrangement in $\PG(N,q)$ of size $N + \left\lfloor N/2\right\rfloor$, provided that $q$ is large enough (see \cite{FancsaliSziklai1} and \cite{Davydov09} for $2\leq N\leq 3$, \cite{BartoliKissetal} for $N=4$, and \cite{Francescoetal} for $N=5$). The weakness of Theorem \ref{HPlines} is that it requires $q$ to be large (both for the construction and for the lower bound), whereas the typical approach in coding theory is to fix $q$ and let the length of the code vary. The only known construction of line sets in higgledy-piggledy arrangement that works for general $N$ and $q$ is the so-called tetrahedron: take $N+1$ points of $\PG(N,q)$ in general position, and then the $\binom{N+1}{2}$ lines joining these points are easily seen to be in higgledy-piggledy arrangement (see \cite{alfarano2019geometric, bonini2020minimal, Davydov09}). However, this construction is much larger than the expected minimum. Also, \cite{Davydov09} gives a slightly smaller $N$-fold strong blocking set for general $N$ and $q$, as well as a general construction for $(t+1)$-fold strong blocking sets in $\PG(N,q)$. 
	
	Thus it is of interest to construct line sets in higgledy-piggledy arrangement in $\PG(N,q)$ of small size from two points of view. First, they give rise to short minimal codes. Second, to determine whether the lower bound remains valid for small $q$ (and possibly large dimension) as well. The proof of our main result Theorem \ref{probab}  relies on a probabilistic construction of higgledy-piggledy line arrangements in $\PG(N,q)$ containing less than $2N-1$ lines (see Theorem \ref{main1}), hence it improves Theorem \ref{HPlines} \textit{ii)}. Furthermore, following this idea, a simple randomised computer search (see Section \ref{sec:randomq2}) provides examples of higgledy-piggledy line sets in $\PG(N,2)$ of size less than $N+\lfloor N/2\rfloor$ for particular small values of $N$, hence we obtain that the lower bound $N+\lfloor N/2\rfloor$ (Theorem \ref{HPlines} \textit{i)}) is not universally valid. On the other hand, we show in Theorem \ref{HP_bound} and  Remark \ref{HPLBimprovement} that the assumption in Theorem \ref{HPlines} \textit{i)} may be relaxed a bit. Let us also remark that the construction behind Theorem \ref{HPlines} and the ones mentioned after it are all based on careful selection of lines, and have an algebraic fashion. In contrast, in the proof of Theorem \ref{main1} we select lines of $\PG(N,q)$ randomly and prove that this results in a higgledy-piggledy line set of size smaller than $2N-1$ with positive probability.
	
	Let us mention that in $\PG(2,q)$, that is, projective planes, cutting blocking sets coincide with \emph{double blocking sets} (point sets which intersect each line in at least two points). For lower bounds on the size of a double blocking set and constructions of the currently known smallest examples, we refer to \cite{BB, BLSSz, CsH, DBHSzVdV}.
	
	Finding short minimal codes, that is, small multifold strong blocking sets, has relevance in another code theoretic aspect as well, since multifold strong blocking sets are linked to {\em covering codes}. From a geometric perspective, these objects correspond to {\em saturating sets} in projective spaces.

	\begin{defi}[Saturating sets]
		A point set $S \subset \PG(N,q)$ is \em{$\rho$-saturating} if for any point $Q$ of $\PG(N,q) \setminus S$ there exist $\rho+1$ points in $S$ generating a subspace of $\PG(N,q)$ which contains $Q$, and $\rho$ is the smallest value with this property. Equivalently, the subspaces of dimension $\rho$ which are generated by the $(\rho+1)$-tuples of $S$ must cover every point of the space. The smallest size of a $\rho$-saturating set in $\PG(N,q)$ is denoted by $s_q(N,\rho)$.
	\end{defi}
	
	\begin{defi}[Covering radius, covering code]
		\label{radius} The \emph{covering radius} of an $[n,n-r]_{q}$
		code is the least integer $R$ such that the space $\F_{q}^{n}$ is covered by spheres of radius $R$ centered on codewords. If an $[n,n-r]_{q}$ code has covering radius $R$, then it is referred to as an $[n,n-r]_{q}R$ covering code.
	\end{defi}
	Note that we can apply the following equivalent description. A linear  code of co-dimension $r$ has \emph{covering radius} $R$ if every (column) vector of $\F_{q}^{r}$ is equal to a linear combination of $R$ columns of a parity check matrix of the code, and $R$ is the smallest value with this property.

	The covering problem for codes is that of finding codes
	with small covering radius with respect to their lengths and
	dimensions. {\em Covering codes} are those codes which are investigated from the point of view of the above covering problem. Usually the parameters for the covering radius and the co-dimension are fixed and one seeks a good upper bound for the length of the corresponding covering codes. 
	
	\begin{defi}
		The length function $l_q(r,R)$ is the smallest length of a $q$-ary linear code of
		co-dimension $r$ and covering radius $R$.
	\end{defi}
	
	There is a one-to-one correspondence between $[n,n-r]_qR$ codes and $(R-1)$-saturating sets of size $n$ in $\PG(r-1,q)$. This implies $l_q(r,R) = s_q(r-1,R-1)$ \cite{Davydov09, DMP}.
	Applying a random construction based on point sets of subspaces in the spirit of higgledy-piggledy line sets (i.e. Theorem \ref{main1}), we improve the known upper bounds when $q$ is an $R$th power and $R\geq \frac{2}{3}r$. Let us note that these results are also related to \emph{subspace designs}; references are given in Section \ref{sec:covsatu}.

	Our paper is organised as follows. In Section \ref{sec:prelim}, we introduce the main notation and recall some basic definitions and propositions. 
	The following Section \ref{sec:geom} is mainly devoted to provide simpler geometric arguments to known bounds on the length of minimal codes, that is, the size of cutting blocking sets. Most of these results were obtained recently in the paper of Alfarano, Borello, Neri and Ravagnani \cite{alfa+neri}. We also derive some quick consequences on higgledy-piggledy line sets. 
	We continue in Section \ref{sec:prob} by a probabilistic argument which largely improves any previously known general upper bounds for cutting blocking sets, minimal codes \cite{alfa+neri, Davydov09} or higgledy-piggledy line sets \cite{FancsaliSziklai1}. Section \ref{sec:prob} mainly deals with the cases where the underlying field has more than $2$ elements, while Section \ref{sec:randomq2} is devoted to the $q=2$ case. These results in turn imply improved results on covering codes and saturating sets as well that we present in Section \ref{sec:covsatu}.

	\section{Preliminaries}\label{sec:prelim}

	The Hamming distance $d(v,c)$ of vectors $v$ and $c$ in
	$\F_{q}^{n}$ is the number of positions in which $v$ and $c$
	differ. The (Hamming) weight $w(c)$ of a vector $c$ is the number of  nonzero coordinates of $c$. 
	The smallest Hamming distance between distinct code
	vectors is called the {\em minimum distance} of the code. An
	$[n,r]_{q}$ code with minimum distance $d$ is denoted as an $
	[n,r,d]_{q}$ code. Note that for a linear code, the minimum distance is equal to the weight of a minimum weight codeword.
	The sphere of radius $R$ with center $c$
	in $\F_{q}^{n}$ is the set $\{v:v\in \F_{q}^{n},$ $d(v,c)\leq
	R\}$.
	
	For a set of points $X$ in $\PG(N,q)$, $\angbra{X}$ denotes the subspace spanned by $X$; that is, the intersection of all subspaces containing $X$.
	
	$\begin{bmatrix}n\\ k\end{bmatrix}_q$ denotes the Gaussian binomial coefficient, whose value  counts the number of subspaces of dimension $k$ in a vector space of dimension $n$ over a finite field with $q$ elements, or likewise, number of subspaces of dimension $k-1$ in a projective space $\PG(n-1, q)$ of dimension $n-1$ over $\F_q$. More precisely,
	
	$$ \begin{bmatrix}n\\ k\end{bmatrix}_q
	= \begin{cases}
		\frac{(q^n-1)(q^{n-1}-1)\cdots(q^{n-k+1}-1)} {(q-1)(q^2-1)\cdots(q^k-1)} & k \le n, \\
		0 & k>n. \end{cases}$$
	
	$\theta_n$ denotes the number of points in $\PG(n,q)$, thus $\theta_n=\sum_{t=0}^n q^t= \begin{bmatrix}n+1\\ 1\end{bmatrix}_q$.
	
	\begin{prop}\label{rajta}
		The number of $m$-dimensional subspaces containing a given $k$-dimensional subspace in $\PG(n, q)$ equals $ \begin{bmatrix}n-k\\ n-m\end{bmatrix}_q$.
	\end{prop}

	\begin{lemma}\label{GBinom-becsles}\mbox{}
		
		$ \begin{bmatrix}n\\ k\end{bmatrix}_q< q^{(n-k)k}\cdot e^{1/(q-2)}$ for $q>2$ and

		$ \begin{bmatrix}n\\ k\end{bmatrix}_2<2^{(n-k)k+1}\cdot e^{2/3}$ for $q=2$.
		
		For the special case $k=n-2$, we have\\
		$ \begin{bmatrix}n\\ n-2\end{bmatrix}_q < q^{2(n-2)}\cdot \frac{q}{q-1}\frac{q^2}{q^2-1}$ for $q\geq2$.
		
	\end{lemma}
	
	Since this lemma is a simple but technical one, we opt to give a proof in the Appendix.

	\begin{defi}
		An $[n,k]_q$ code $\cC$ is \emph{non-degenerate}, if there is no $i\in\{1,\ldots,n\}$ such that $c_i=0$ for all $c\in\cC$. An $[n,k]_q$ code $\cC$ is \emph{projective}, if the coordinates of the codewords are linearly independent; that is, there exists no $i\neq j\in\{1,\ldots,n\}$ and $\lambda\in\F_q^*$ such that $c_i=\lambda c_j$ for every codeword $c\in\cC$. 
	\end{defi}
	
	In terms of the generator matrix $G$ of $\cC$, non-degeneracy means that every column of $G$ is nonzero, and projectivity means that no column of $G$ is a scalar multiple of any other column of $G$.
	
	\begin{defi}
		Let $\cC$ be an $[n,k]_q$ code. The \emph{support} $\sigma(c)$ of a codeword $c$ is the set of nonzero coordinates of $c$; that is, $\sigma(c)=\{i\colon c_i\neq 0\}$. A codeword $c\in\cC$ is \emph{minimal} if for every $c'\in\cC$ we have $\sigma(c')\subseteq\sigma(c)$ if and only if $c'=\lambda c$ for some $\lambda\in\F_q^*$. A codeword $c\in\cC$ is \emph{maximal} if for every $c'\in\cC$ we have $\sigma(c')\supseteq\sigma(c)$ if and only if $c'=\lambda c$ for some $\lambda\in\F_q^*$. The code $\cC$ is \em{minimal}, if each codeword of $\cC$ is minimal.
	\end{defi}
	
	Note that for each codeword $c$, $|\sigma(c)| = w(c)$. A maximum (minimum) weight codeword may not be maximal (minimal) and, also, a maximal (minimal) codeword is not necessarily a maximum (minimum) weight codeword.

	\section{Geometrical arguments and higgledy-piggledy line sets}\label{sec:geom}
	
	\subsection{Geometrical arguments}
	
	In this section we aim to emphasise the geometrical interpretation of (minimal) codes in order to apply finite geometrical tools in their analysis. This was done in \cite{alfa+neri, alfarano2019geometric, lu2019parameters, tang2019full} as well, for example, but our intention is to use finite geometrical arguments much more transparently. Most results presented here are found in \cite{alfa+neri}; however, we believe that the usefulness of our approach is justified by the simplicity of the proofs.
	
	A well-known and often exploited interpretation of linear codes is the following. Let $G$ be a generator matrix of a non-degenerate $[n,k]_q$ code $\cC$. Then the columns $G_1,\ldots,G_n$ of $G$ may be interpreted as points of the projective space $\PG(k-1,q)$. Let $S(G)=\{G_1,\ldots,G_n\}$ denote the (multi)set of points in $\PG(k-1,q)$ corresponding to $\cC$. Clearly, different generator matrices of $\cC$ yield projectively equivalent (multi)sets of $\PG(k-1,q)$. From now on, we will always assume that a generator matrix $G$ of a given code $\cC$ is fixed, and it will not cause ambiguity to omit the references for the generator matrix $G$ and, e.g., write only $S$. Note that $S$ is a set if and only if $\cC$ is projective. 
	
	\begin{notation}
		Let $u\in\F_q^k$. Then let $\Lambda_u=\{x\in\PG(k-1,q)\colon x\perp u\}$ be a hyperplane of $\PG(k-1,q)$. For a point set $S$ of $\PG(k-1,q)$, $S_u$ denotes $S\setminus \Lambda_u$. For a vector $v$ over $\F_q$, let $\angbra{v}=\{\lambda v\colon \lambda\in\F_q^*\}$.
	\end{notation}

	Each codeword $c\in\cC$ can be uniquely obtained in the form $uG$ for some $u\in\F_q^k$. Thus we may associate the hyperplane $\Lambda_u$ to $u$. Note that $\Lambda_u=\Lambda_{u'}$ if and only if $\angbra{u}=\angbra{u'}$ if and only if $\angbra{uG}=\angbra{u'G}$; that is, each hyperplane of $\PG(k-1,q)$ corresponds to a set of $q-1$ (nonzero) codewords of $\cC$, which form the nonzero vectors of a one-dimensional linear subspace of $\F_q^n$.

	\begin{lemma}\label{minmax}
		Let $\cC$ be an $[n,k]_q$ code with generator matrix $G=(G_1,\ldots,G_n)$, and let $c=uG\in\cC$, $u\in\F_q^k$. Let $S=\{G_1,\ldots,G_n\}$ be the corresponding point set of $\PG(k-1,q)$. Then 
		\begin{itemize}
			\item $c$ is a minimal codeword if and only if $\angbra{\Lambda_u\cap S}=\Lambda_u$ in $\PG(k-1,q)$;
			\item $c$ is a maximal codeword if and only if $S_u$ intersects every hyperplane of $\PG(k-1,q)$ different from $\Lambda_u$; that is, $S_u$ is an affine blocking set (with respect to hyperplanes) in $\PG(k-1,q)\setminus\Lambda_u \simeq \AG(k-1,q)$.
		\end{itemize}
	\end{lemma}
	\begin{proof}
		For any two codewords $c=uG$ and $c'=u'G$, we clearly have $\sigma(u'G)\subseteq\sigma(uG)$ if and only if $S_{u'}\subseteq S_u$ if and only if $\Lambda_u\cap S \subseteq \Lambda_{u'}\cap S$. 
		
		Fix now $c=uG$. Suppose now that $\angbra{\Lambda_u\cap S}$ is contained in a $2$-codimensional subspace of $\PG(k-1,q)$. Consider a hyperplane $\Lambda_{u'}\neq\Lambda_u$ containing this subspace. Then, clearly, $\angbra{u}\neq\angbra{u'}$ and $S_{u'}\subseteq S_u$, so $c$ is not minimal. On the other hand, if $\Lambda_u\cap S$ spans $\Lambda_u$, then $\sigma(u'G)\subseteq\sigma(uG)$ yields $\Lambda_u\cap S \subseteq \Lambda_{u'}\cap S$, whence $\Lambda_u=\Lambda_{u'}$ and $\angbra{u}=\angbra{u'}$, so $c=uG$ is minimal.
		
		If for some $u'\in\F_q^k$ we have $\angbra{u}\neq\angbra{u'}$, then $\sigma(uG)\not\subseteq\sigma(u'G)$ is equivalent to $\Lambda_{u'}\cap S \not\subseteq \Lambda_{u}\cap S$, which holds if and only if $\Lambda_{u'}$ contains a point of $S\setminus\Lambda_u=S_u$. That is, $c=uG$ is maximal if and only if $S_u$ is an affine blocking set in $\PG(k-1,q)\setminus\Lambda_u$.
	\end{proof}
	
	\begin{cor}\label{mincodecutbl}
		Let $\cC$ be a non-degenerate $[n,k]_q$ code with generator matrix $G=(G_1,\ldots,G_n)$. Let $S=\{G_1,\ldots,G_n\}$ be the corresponding point set of $\PG(k-1,q)$. Then $\cC$ is a minimal code if and only if $S$ is a cutting blocking set.
	\end{cor}
	\begin{proof}
		This follows immediately from Definitions \ref{def:mincode}, \ref{strong_def} and Lemma \ref{minmax}.
	\end{proof}
	
	The above corollary was pointed out in \cite{alfarano2019geometric} and \cite{tang2019full} as well. Furthermore, as it was observed in \cite{alfa+neri}, a code $\cC$ is a minimal code if and only if each codeword of $\cC$ is maximal. In terms of cutting blocking sets, we get the following equivalent description.
	
	\begin{proposition}\label{strblsetequiv1}
		$S$ is a cutting blocking set if and only if for any hyperplane $\Lambda$, the point set $S\setminus\Lambda\subseteq \PG(N,q)\setminus\Lambda\simeq\AG(N,q)$ is an affine blocking set in $\PG(N,q)\setminus\Lambda$ (with respect to hyperplanes).
	\end{proposition}
	\begin{proof}
		This is easy to prove directly; but also follows immediately from Corollary \ref{mincodecutbl}, Lemma \ref{minmax}, and the above observation about the maximality of the codewords of a minimal code.
	\end{proof}
	
	Let us give yet another straightforward but useful description of cutting blocking sets, with which one may also consider them as blocking sets (transversals) of certain hypergraphs.
	
	\begin{proposition}\label{strblsetequiv2}
		Let $\cP$ denote the set of points of $\PG(N,q)$, and let
		\[
		\mathcal{H}=\{T\subseteq\cP\mid \exists H, H' \mbox{ subspaces s.t.\ } \dim{H}=N-1, \dim{H'}=N-2, H'\subset H, T=H\setminus H'\}.
		\]
		Then $S$ is a cutting blocking set of $\PG(N,q)$ if and only if $S\cap T\neq\emptyset$ for all $T\in\mathcal{H}$.
	\end{proposition}
	\begin{proof}
		This is just a reformulation of Proposition \ref{strblsetequiv1}.
	\end{proof}
	
	As seen above, affine blocking sets are tightly connected to cutting blocking sets, hence it is useful to recall a fundamental result on their sizes.
	
	\begin{theorem}[Jamison \cite{Jamison}, Brouwer--Schrijver \cite{BrouwerSchrijver}]\label{Jamison}
		Suppose that $\cB$ is a blocking set of $\AG(N,q)$; that is, $\cB$ is a set of points which intersects each hyperplane of $\AG(N,q)$. Then $|\cB|\geq N(q-1)+1$.
	\end{theorem}
	
	As an immediate consequence, we get the following.
	
	\begin{theorem}[Theorem 2.8 of \cite{alfa+neri}]\label{maxcwweight}
		Let $\cC$ be an $[n,k]_q$ code. If $c\in\cC$ is maximal, then $w(c)\geq (k-1)(q-1)+1$. Thus the minimum weight of a minimal $[n,k]_q$ code is at least $(k-1)(q-1)+1$.
	\end{theorem}
	\begin{proof}
		Let $S$ be the point set of $\PG(k-1,q)$ associated to $\cC$ via the generator matrix $G$. Let $c=uG$. By Lemma \ref{minmax}, $S_u$ is an affine blocking set in $\PG(k-1,q)\setminus\Lambda_u\simeq\AG(k-1,q)$. By Jamison's Theorem, $w(u)=|S_u|\geq (q-1)(k-1)+1$.
	\end{proof}
	
	\begin{cor}[\cite{alfa+neri, alfarano2019geometric,  tang2019full}]\label{corsize}
		Let $\cC$ be a minimal $[n,k]_q$ code. Then $n\geq q(k-1)+1$.
	\end{cor}
	\begin{proof}
		Let $S$ be the point set of $\PG(k-1,q)$ associated to $\cC$ via the generator matrix $G$. Take a codeword $c=uG$ in $\cC$. As $c$ is maximal, $w(c)=|S_u|\geq(q-1)(k-1)+1$. As $c$ is minimal, $\Lambda_u\cap S$ spans $\Lambda_u$, hence $|\Lambda_u\cap S|\geq k-1$. Thus $n=|S|\geq (q-1)(k-1)+1 + (k-1)$.
	\end{proof}
	
	The proof of Theorem \ref{maxcwweight} in \cite{alfa+neri} uses the Alon-F\"uredi Theorem, while \cite{tang2019full} also uses Jamison's theorem to derive Corollary \ref{corsize}. For connections among the Alon-F\"uredi Theorem, Jamison's Theorem and the (punctured) Combinatorial Nullstellensatz, we refer to \cite{ballserra}.

	In what follows, we give a compact geometrical proof of Theorem 2.14 of \cite{alfa+neri}. Let us first formalise the statement in terms of cutting blocking sets.
	
	\begin{theorem}[Theorem 2.14 of \cite{alfa+neri}]\label{cuttinglb}
		A cutting blocking set in $\PG(N,q)$ contains at least $N(q+1)$ points.
	\end{theorem}
	\begin{proof}
		Let $S$ be a cutting blocking set, and let $H$ be a hyperplane for which $|H\cap S|$ is maximal. By Proposition \ref{strblsetequiv1}, $S_{H}=S\setminus H$ is an affine blocking set in $\PG(N,q)\setminus H$ (with respect to hyperplanes). Let $S_H'$ be a minimal affine blocking set inside $S_H$ (with respect to set theoretical containment), and let $P\in S_H'$. By Theorem \ref{Jamison}, $|S_H'|\geq N(q-1)+1$. Since $S_H'\setminus\{P\}$ is not an affine blocking set, there exists a hyperplane $U$ for which $S_H'\cap U = \{P\}$. Then $|S_H\setminus U|\geq |S_H'\setminus U|\geq N(q-1)$. By the pigeonhole principle, one of the $q-1$ hyperplanes containing $H\cap U$ different from $H$ and $U$ contains at least $N$ points of $|S_H\setminus U|$. Let $Z$ be such a hyperplane. Then, by the choice of $H$, we have 
		\begin{eqnarray*}
			|H\cap S|\geq |Z\cap S| &=& |Z \cap S_H| + |Z \cap (S \cap H)| = |Z \cap (S_H\setminus U)| + |U \cap (S \cap H)| \\
			&=& |Z \cap (S_H\setminus U)| + |S\cap H|+ |(S\setminus H)\setminus U| - |S\setminus U| \\
			&\geq& N + |S\cap H| + N(q-1) - |S\setminus U|,
		\end{eqnarray*}
		whence $|S\setminus U| \geq Nq$. As $S\cap U$ spans $U$, $|S\cap U|\geq N$, thus $|S| = |S\setminus U| + |S\cap U|\geq N(q+1)$.
	\end{proof}
	
	Let us remark that if $N\leq q$ holds, then this result follows immediately from the well-known fact that under the assumption $N\leq q$, any $t$-fold $1$-blocking set of $\PG(N,q)$ contains at least $t(q+1)$ points (recall that a cutting blocking set is an $N$-fold $1$-blocking set as well). However, Theorem \ref{cuttinglb} makes no assumption on $q$ or $N$, which makes it much more useful from the coding theoretic aspect. Let us formulate the coding theoretic version of Theorem \ref{cuttinglb}.
	
	\begin{theorem}[Theorem 2.14 of \cite{alfa+neri}]
		Let $\cC$ be a minimal $[n,k]_q$ code. Then $n\geq (q+1)(k-1)$.
	\end{theorem}

	\subsection{Notes on higgledy-piggledy line sets}
	
	In the following, let us slightly strengthen Theorem \ref{HPlines} on the size of higgledy-piggledy line sets.
	
	\begin{lemma}\label{HPlines-hyperplane}
		Let $\cL$ be a higgledy-piggledy line set of $\PG(N,q)$, and suppose that a hyperplane $H$ contains $t$ lines of $\cL$. Then $|\cL|\geq N + t - \left\lfloor\frac{N-1}{q}\right\rfloor$.
	\end{lemma}
	\begin{proof}
		Since $S=\cup_{\ell\in\cL}\ell$ is a cutting blocking set, $S\setminus H$ is an affine blocking set, and thus it contains at least $N(q-1)+1$ points (see Jamison's Theorem \ref{Jamison}). Since each line of $\cL$ contains at most $q$ points of $S\setminus H$, we need at least $\frac{|S\setminus H|}{q}\geq N - \frac{N-1}{q}$ lines to cover the points of $S\setminus H$. Together with the lines contained in $H$, this proves the assertion.
	\end{proof}
	
	\begin{theorem}\label{HP_bound}
		A line set of $\PG(N,q)$ in higgledy-piggledy arrangement contains at least $N + \left\lfloor \frac{N}{2}\right\rfloor - \left\lfloor\frac{N-1}{q}\right\rfloor$ elements.
	\end{theorem}
	\begin{proof}
		Let $\cL$ be a higgledy-piggledy line set of $\PG(N,q)$, and let $S=\cup_{\ell\in\cL}\ell$.
		As $x$ lines of $\PG(N,q)$ span a subspace of dimension at most $2x-1$, we can find a hyperplane $H$ which contains at least $\left\lfloor\frac{N}{2}\right\rfloor$ lines of $\cL$ (we need that $|\cL|\geq \left\lfloor\frac{N}{2}\right\rfloor$, but this is clear since $\cL$ cannot be contained in a hyperplane). Apply Lemma \ref{HPlines-hyperplane} to finish the proof.
	\end{proof}
	
	\begin{remark}\label{HPLBimprovement}
		Note that if $q\geq N$, the above result yields that a higgledy-piggledy line set contains at least $N+\left\lfloor\frac{N}{2}\right\rfloor$ lines. Comparing this to  the assumption $q\geq N+\left\lfloor \frac{N}{2}\right\rfloor$ of Theorem \ref{HPlines}, one can see that Theorem \ref{HP_bound} gives a strengthening of the theorem of Fancsali and Sziklai.
	\end{remark}

	Fancsali and Sziklai prove that if there is no $(N-2)$-dimensional subspace intersecting every line of a line set in $\PG(N,q)$, then the line set is higgledy-piggledy \cite[Theorem 11]{FancsaliSziklai1}, and they also prove that if $q>|\cL|$ for a higgledy-piggledy line set $\cL$, then $\cL$ has the aforementioned property \cite[Lemma 12]{FancsaliSziklai1}. Thus they call this property `almost equivalent' to being higgledy-piggledy. In the light of these considerations, it might be somewhat surprising that minimal higgledy-piggledy line sets \emph{always} admit an $(N-2)$-dimensional subspace which intersects all but possibly one of their lines. 
	
	\begin{proposition}\label{minHP}
		Suppose that $\cL=\{\ell_1,\ldots,\ell_m\}$ is a minimal set of higgeldy-piggledy lines in $\PG(N,q)$ (that is, $S=\cup_{i=1}^m \ell_i$ is a cutting blocking set of $\PG(N,q)$, but for all $j\in\{1,\ldots,m\}$, $S_j:=\cup_{\substack{i=1\\i\neq j}}^m \ell_i$ is not a cutting blocking set). Then for all $j\in\{1,\ldots,m\}$, there exists a subspace $\Lambda_j$ of co-dimension $2$ which intersects $\ell_i$ for each $i\in\{1,\ldots,m\}\setminus\{j\}$. Moreover, there exists a hyperplane $H_j$ containing $\Lambda_j$, which contains only those lines of $\cL\setminus\{\ell_j\}$ that are contained in $\Lambda_j$.
	\end{proposition}
	\begin{proof}
		Fix $j\in\{1,\ldots,m\}$. Since $S_j$ is not a cutting blocking set, there exist a hyperplane $H_1$ for which the point set $S_j\setminus H_1$ is not an affine blocking set in $\PG(N,q)\setminus H_1$, that is, there exists another hyperplane $H_2$ such that $H_2\cap (S_j\setminus H_1)=\emptyset$. As each line intersects $H_2$, this means that for each $i\in\{1,\ldots,m\}\setminus\{j\}$, $\ell_i\cap H_2\in H_1$. Thus the subspace $H_1\cap H_2$ and the hyperplane $H_2$ are appropriate choices to prove the assertion.
	\end{proof}
	
	Note that it might occur that a subspace of co-dimension $2$ blocks every line of a higgledy-piggledy line set $\cL$ of $\PG(N,q)$. In fact, this is the case whenever our line set has at most $N+\lfloor N/2\rfloor -1$ elements \cite[Lemma 13]{FancsaliSziklai1}, in which case $|\cL|\geq q+1$  holds \cite[Lemma 12]{FancsaliSziklai1}; but, if $q$ is small compared to $N$, the latter conclusion is meaningless. In particular, when $q=2$, there exist examples of higgledy-piggledy line sets of $\PG(N,2)$ of size less than $N+\lfloor N/2\rfloor$, see Section \ref{sec:randomq2}. As seen, these must admit a subspace of dimension $(N-2)$ intersecting all their lines.
	
	Let us point out that if $q$ is small, then, by the pigeonhole principle, Proposition \ref{minHP} yields that the $q+1$ hyperplanes passing through the $(N-2)$-dimensional subspace found therein behave unbalanced regarding the number of lines of $\cL$ they contain.

	\section{Probabilistic approach}\label{sec:prob}
	
	\begin{theorem}\label{main1}
		There exists a cutting blocking set in $\PG(N,q)$ of size at most $m(q+1)$ which consists of the points of at most $m$ lines, where 
		\[m = 
		\left\{
		\begin{matrix}
			\left\lceil \frac{2}{1+\frac{1}{\ln(q)(q+1)^2}}N\right\rceil & \normalsize\mbox{if } q>2, \\
			& \\
			\left\lceil1.95N\right\rceil & \mbox{if } q=2.
		\end{matrix}
		\right.\]
	\end{theorem}
	
	In other words, this theorem ensures the existence of $m$ lines in higgledy-piggledy arrangement. Note that the multiplier of $N$  is strictly smaller than $2$.

	\begin{proof}
		Let us take a projective space $\PG(N,q)$, and choose $m$ lines, $\ell_1, \ell_2, \ldots \ell_{m}$  uniformly at random. We denote  this multiset by $\mathcal{L}=\{\ell_1, \ell_2, \ldots \ell_{m}\}$.
		
		Our aim is to bound from below the probability that the point set  $\B= \bigcup_{i=1}^{m} \ell_i$ intersects every hyperplane  $H$ in a subset which spans $H$ itself.
		
		Clearly, by the definition of cutting blocking sets,
		\[
		\PP(\B \mbox{ is a cutting blocking set}) =
		1-\PP(\,\exists\, H \colon \dim H =N-1, \ \langle  H \cap \B\rangle \neq  H).
		\]
		
		Observe that for a hyperplane $H$, $\langle  H \cap \B\rangle \neq  H$ implies that $\dim\langle  H \cap \B\rangle < N-1$. Since every  $\ell_i$ intersects $H$, the intersections must be covered by a subspace $\Lambda$ of dimension at most $N-2$. Clearly, if there exists a subspace of dimension at most $N-3$ intersecting every $\ell_i$, then there exists such a subspace of dimension exactly $N-3$ as well. Moreover, if the dimension of a covering subspace $\Lambda$ is $N-2$, then none of the lines $\ell_i$ intersects $H\setminus \Lambda$.
		From this, we obtain the following bound.
		
		\begin{equation}\label{eq:probs}
			\begin{split}
				\PP(\B \mbox{ is a cutting blocking set}) \geq 
				1-\PP(\,\exists\,\Lambda \colon \dim \Lambda=N-3, \ \forall i\:\ell_i\cap \Lambda\neq \emptyset)\\
				\quad
				- 
				\PP(\,\exists\,\Lambda \colon \dim \Lambda=N-2, \ \forall i\:\ell_i\cap \Lambda\neq \emptyset, \ 
				\bigcup_i \langle (\B\cap\Lambda) \cup \ell_i\rangle \neq  \PG(N,q) ).
			\end{split}
		\end{equation}

		Firstly we give a bound to the term $\PP(\,\exists\,\Lambda \colon \dim \Lambda=d, \ \forall i\:\ell_i\cap \Lambda\neq \emptyset)$ with $d\leq N-2$ in Inequality \eqref{eq:probs}.
		
		\begin{equation}\label{eq:prob_kicsi}
			\begin{split}
				\PP(\exists \ \Lambda \colon \dim \Lambda=d, \ \forall i\: \ell_i\cap \Lambda\neq \emptyset)\leq 
				\begin{bmatrix}N+1\\ d+1\end{bmatrix}_q   \cdot
				\PP(\ell_1 \cap \Lambda\neq \emptyset)^{m},
			\end{split}
		\end{equation} 
		where the last probability is taken for a fixed $d$-dimensional subspace $\Lambda$ and the line $\ell_1$ chosen uniformly at random. Distinguishing the lines which are contained in $\Lambda$ from those which intersect $\Lambda$ in a single point, we get the formula below. (For precise details, see the Appendix.)
		\begin{equation}\label{nemures}
			\begin{split}
				\PP(\ell_1 \cap \Lambda\neq \emptyset)= \frac{ \begin{bmatrix}d+1\\ 2\end{bmatrix}_q   + \begin{bmatrix}d+1\\ 1\end{bmatrix}_q  \cdot \frac{1}{q} \left( \begin{bmatrix}N+1\\ 1\end{bmatrix}_q -\begin{bmatrix}d+1\\ 1\end{bmatrix}_q       \right)     }{ \begin{bmatrix}N+1\\ 2\end{bmatrix}_q   }\\
				< q^{d-N+1}+q^{d-N}-q^{2d-2N+1}.
			\end{split}
		\end{equation} 
		
		By combining Inequalities \eqref{eq:prob_kicsi} and \eqref{nemures} with the upper bound of Lemma \ref{GBinom-becsles} on the Gaussian binomial coefficients, we get
		%
		\begin{equation}\label{eq:intermed0}
			\begin{split}
				\PP(\,\exists\, \Lambda \colon \dim \Lambda=N-3, \ \forall i\:\ell_i\cap \Lambda\neq \emptyset)<q^{3(N-2)-2m}\cdot\left(1+\frac{1}{q}-\frac{1}{q^3}\right)^{m}\cdot\gamma(q),
			\end{split}
		\end{equation}
		where  $\gamma(q)= e^{1/(q-2)}$ for $q>2$ and $\gamma(q)= 2e^{2/3}$ for $q=2$.

		It is easy to see that if $q\geq 3$, then  $q^{3N-2m}\cdot\left(1+\frac{1}{q}\right)^{m}<1$ holds for $m\geq 1.8N$, hence 
		
		\begin{equation}\label{eq:intermed2}
			\begin{split}
				p_{<N-2}:=\PP(\,\exists\,\Lambda \colon \dim \Lambda=N-3, \ \forall i\:\ell_i\cap \Lambda\neq \emptyset)<q^{-6}\cdot\gamma(q)
			\end{split}
		\end{equation} in the case $q\geq 3$. Case $q=2$  allows us to choose 
		$m=\left\lceil\frac{\ln{8}}{\ln{(32/11)}}N\right\rceil\leq \left\lceil1.95N\right\rceil$ in order to get 
		$2^{3N-2m}\cdot\left(1+\frac{1}{2}-\frac{1}{8}\right)^{m}\leq 1$, which implies Inequality \eqref{eq:intermed2} similarly for $q=2$. 

		We continue by estimating the final summand, namely \[
		\PP(\,\exists\,\Lambda \colon \dim \Lambda=N-2, \ \forall i\:\ell_i\cap \Lambda\neq \emptyset, \ 
		\bigcup_i \langle (\B\cap\Lambda) \cup \ell_i\rangle \neq  \PG(N,q) ).
		\]
		Suppose now that every line of $\cL$ intersects a fixed subspace $\Lambda$ of dimension $N-2$. 
		We have 
		\[
		\eta:=\PP(\ell \subseteq \Lambda \ \mid \ \ell\cap \Lambda\neq \emptyset )= \frac{\begin{bmatrix}N-1\\ 2\end{bmatrix}_q}{\begin{bmatrix}N-1\\ 2\end{bmatrix}_q + (q^{N-1}+q^{N-2})\begin{bmatrix}N-1\\ 1\end{bmatrix}_q} <\frac{1}{q^3+q^2-q}.
		\]
		Observing that there are $q+1$ hyperplanes through $\Lambda$, the probability of the (bad) event that every line $\ell\in \cL$ is either included in $\Lambda$ or not included in a fixed hyperplane through $\Lambda$ can be bounded above by the following formula:
		\[ 
		\left(\eta+(1-\eta)\frac{q}{q+1}\right)^m<\left(\frac{1}{q^3+q^2-q}+\left(1-\frac{1}{q^3+q^2-q}\right)\frac{q}{q+1}\right)^{m} < \left(\frac{q}{q+1}+\frac{1}{q^3(q+1)}\right)^{m}.
		\]
		Adding the probability of these events for every hyperplane through $\Lambda$, we obtain
		\begin{equation}\label{eq:main_term}
			\begin{split}
				p_{N-2}:=\PP(\,\exists\,\Lambda \colon \dim \Lambda=N-2, \ \forall i\:\ell_i\cap \Lambda\neq \emptyset, \ 
				\bigcup_i \langle (\B\cap\Lambda) \cup \ell_i\rangle \neq  \PG(N,q) )\leq\\
				\PP\left(\,\exists\,\Lambda \colon \dim \Lambda=N-2, \ \forall i\:\ell_i\cap \Lambda\neq \emptyset\right)\cdot \left( (q+1) \left(\frac{q}{q+1}+\frac{1}{q^3(q+1)}\right)^{m} \right)
			\end{split}
		\end{equation}
		via conditional probability.
		We use Inequality \eqref{eq:prob_kicsi}, Lemma \ref{GBinom-becsles} for the $2$-codimenional case, and Inequality \eqref{nemures} 
		and apply similar calculations to those in Inequalities \eqref{eq:intermed0}  but with $d=N-2$. This provides
		\begin{equation}\label{eq:4.7}
			\begin{split}
				p_{N-2}\leq {q^{2(N-1)-m}}\cdot \frac{q}{q-1}\frac{q^2}{q^2-1}\left(1+\frac{1}{q}-\frac{1}{q^2}\right)^{m} \left(( q+1)\cdot\left(\frac{q}{q+1}+\frac{1}{q^3(q+1)}\right)^{m}\right)\\
				<\frac{q}{(q-1)^2}{q^{2N-m}}\left(1-\frac{1}{(q+1)^2}\right)^{m} .
			\end{split}
		\end{equation}
		
		Finally, suppose first that $q>2$.
		Putting all these estimates together, we get that if $m\geq 1.8N$, then
		\begin{eqnarray}
			\PP(\B \mbox{ is a cutting blocking set}) &\geq& 1 - p_{<N-2} - p_{N-2}\\
			&>& 1 - q^{-6}\cdot\gamma(q) - \frac{q^{2N-m+1}}{(q-1)^2}\left(1-\frac{1}{(q+1)^2}\right)^{m}.\label{eq:4.7i}
		\end{eqnarray}
		This event is of positive probability for the point set of $m= \left\lceil\frac{2N}{1+\frac{1}{(q+1)^2\ln(q)}}\right\rceil$ randomly chosen lines.

		If $q=2$, we apply a stronger form of Inequality \ref{eq:4.7i} by taking $\left(\frac{1}{q^3+q^2-q}+\left(1-\frac{1}{q^3+q^2-q}\right)\frac{q}{q+1}\right)$ instead of its upper bound estimate  $\left(\frac{q}{q+1}+\frac{1}{q^3(q+1)}\right)$ from Inequality \ref{eq:4.7}.
		
		Then 
		
		\begin{eqnarray*}
			\PP(\B \mbox{ is a cutting blocking set}) &\geq& 1 - p_{<N-2} - p_{N-2}\\
			&>& 1 - 2^{-5}\cdot e^{2/3}- {2^{2N-m+1}}\left(\frac{115}{132}\right)^{m}\\
			&>&0.\label{eq:4.10}
		\end{eqnarray*}
		holds for $m\geq 1.95N$ when $N>2$, while the statement is trivial for $N=2$.
	\end{proof}

	\section{Random constructions in $\PG(N,2)$} \label{sec:randomq2}
	
	When the order of the field is two or, in other words, in the case of binary minimal codes, we provide a better upper bound on the size of cutting blocking sets than the one in Theorem \ref{main1}. Let us mention that in this case, minimal codes coincide with the so-called intersecting codes. For more information about intersecting codes, we refer the reader to \cite{chabanne2013towards} and \cite{Cohen-Lempel}. Recall that for $q=2$, $\theta_i=2^i+2^{i-1}+\ldots+1=2^{i+1}-1$.
	
	\subsection{Uniform random point sets}
	
	Take a set $S$ of $x$ points of $\PG(N,2)$ uniformly at random. By Proposition \ref{strblsetequiv2}, $S$ is a cutting blocking set if and only if it intersects each element of 
	\[
	\mathcal{H}=\{T\subseteq\cP\mid \exists H, H' \mbox{ subspaces s.t } \dim{H}=N-1, \dim{H'}=n-2, H'\subset H, T=H\setminus H'\},
	\]
	where $\cP$ denotes the point set of $\PG(N,q)$. Clearly, $|\mathcal{H}|=\theta_N\theta_{N-1}$, and for each $T\in\mathcal{H}$, $|T|=\theta_{N-1}-\theta_{N-2}=2^{N-1}$. Consequently, the probability that a set $T\in\mathcal{H}$ is missed by $S$ is $\left(\frac{\theta_N-2^{N-1}}{\theta_N}\right)^x<(\frac{3}{4})^x$, which yields the following bound on the expected value of the number of elements of $\mathcal{H}$ not intersecting $S$:
	\[
	\E(T\in\mathcal{H} \ | \  S\cap T=\emptyset)< 2^{2N+1}\left(\frac{3}{4}\right)^x.
	\]
	
	The existence of a cutting blocking set of size $x$ follows if the latter formula is less than $1$, thus $x=\left\lceil\frac{\log(2)}{\log (4/3)}(2N+1) \right\rceil \approx \left\lceil2.41\cdot (2N+1) \right\rceil$ suffices. Thus we have the following result.
	
	\begin{theorem}\label{thm:binary}
		In $\PG(N,2)$, there exists a cutting blocking set of size 
		\[
		\left\lceil\frac{\log(2)}{\log (4/3)}(2N+1) \right\rceil.
		\]
	\end{theorem}
	
	Note that this random construction improves the bound of Theorem \ref{Cohen} by an additive constant. 
	
	\begin{cor}\label{cor:binary}
		\[m(k,2)\leq \frac{2k-1}{\log_2(\frac{4}{3})}.\]
	\end{cor}


	
	

	\subsection{Explicit results for small $N$}
	
	We have utilised a computer to perform a simple Monte Carlo search.
	First we chose a set $S$ of $x$ points of uniform random distribution and then checked if the result was a cutting blocking set of $\PG(N,q)$. We tried to decrease $x$ as much as possible. The sizes of the smallest cutting blocking sets found this way are found in the next table.
	
	\begin{center}
		\begin{tabular}{c|ccccccccc}
			$N$          & 2 & 3 & 4 & 5 & 6 & 7 & 8 & 9 & 10 \\ 
			\hline
			$|S|$         & 6 & 9 & 13& 17& 22& 27& 32& 37& 44         
		\end{tabular}
	\end{center}
	
	Doing the same in order to find $m$ lines in higgledy-piggledy arrangement in $\PG(N,q)$, we obtained the results shown in the next table. For the sake of easy comparison with the bound $m\geq \lfloor 3N/2 \rfloor$ known to be valid for $q$ large enough (cf.\ Theorem \ref{HPlines} and Remark \ref{HPLBimprovement}), we inserted the value of $\lfloor 3N/2 \rfloor$ as well. Also, as the elements of a line set $\{\ell_1,\ldots,\ell_m\}$ in higgledy-piggledy arrangement are not necessarily disjoint, their union may be smaller than $m(q+1)$.
	
	\begin{center}
		\begin{tabular}{c|ccccccccc}
			$N$          & 2 & 3 & 4 & 5 & 6 & 7 & 8 & 9 & 10 \\ 
			\hline
			$\left\lfloor\frac{3N}{2}\right\rfloor$ & 
			3 & 4 & 6 & 7 & 9 & 10& 12& 13& 15 \\
			$m$           & 3 & 4 & 5 & 6 & 8 &  9& 11& 13& 14 \\          
			$|\cup_{i=1}^m \ell_i|$ & 
			6 &  9& 13& 18& 23& 27& 32& 38& 42         
		\end{tabular}
	\end{center}

	\section{Covering codes and saturating sets}\label{sec:covsatu}
	

	As we explained in detail in the Introduction, saturating sets and covering codes are corresponding objects, and bounding the size of a saturating set corresponds to bounding the length function $l_q(r,R)$ of the covering code. From now on, we use the geometric terminology.
	Let us recall the concept of $\varrho$--saturating sets of a projective plane $\PG(N, q)$.
	
	\begin{defi}
		A set $\cS$ of points of $\PG(N, q)$ is said to be {\em $\varrho$--saturating} if for any point $P \in \PG(N, q)$ there exist $\varrho + 1$ points of $\cS$ spanning a subspace of $\PG(N, q)$ containing $P$, and $\varrho$ is the smallest value with such property.
	\end{defi}
	
	Davydov, Giulietti, Marcugini and Pambianco  proved \cite{Davydov09}  a key connection between $(\varrho+1)$-fold strong blocking sets (or cutting $(N-\varrho)$ blocking sets) and $\varrho$-saturating sets.
	
	\begin{theorem}[\cite {Davydov09}, Theorem 3.2.]\label{kapcs} 
		Any $(\varrho+1)$-fold strong blocking set in a subgeometry $\PG(N, q)\subset \PG(N, q^{\varrho+1})$ is a $\varrho$-saturating set in the space $\PG(N,q^{\varrho+1})$.
	\end{theorem}

	Let $s_q(N, \varrho)$ denote the smallest size of a $\varrho$--saturating set of $\PG(N, q)$.
	For recent upper bounds on $\varrho$--saturating sets of $\PG(N, q)$ the reader is referred to \cite{DMP, Denaux}. 
	
	\begin{theorem}[Denaux \cite{Denaux}, Theorem 6.2.12.]\label{Den} Suppose that $q$ is a prime power. Then
		$$\frac{\varrho+1}{e}q^{N-\varrho}<   s_{q^{\varrho+1}}(N, \varrho)\leq \frac{(\varrho+1)(\varrho+2)}{2} \left( {q^{N-\varrho}} +\frac{2\varrho}{\varrho+2}\frac{q^{N-\varrho}-1}{q-1} \right).$$
	\end{theorem}

	The most well-studied case is $\varrho=1$ where Theorem \ref{kapcs} provides an upper bound on saturating sets via the size of $2$-fold strong blocking sets, while the other side of the spectrum, namely the case of $N$-fold strong blocking sets (that we called cutting blocking sets for brevity) is also significant. Our probabilistic upper bound in Theorem \ref{main1}  thus in turn gives the following corollary.
	
	\begin{cor} 
		If $q>2$, then
		$s_{q^{N}}(N, N-1)\leq \left\lceil \frac{2N}{1+\frac{1}{(q+1)^2\ln(q)}}\right\rceil(q+1)$.
	\end{cor}
	
	Note that the previously known general result in this direction was the tetrahedron construction, see \cite{alfa+neri} which provides a point set of size $\binom{N+1}{2}(q-1)+N+1$ for a cutting blocking set, and also for a saturating set in  $\PG(N, q^2)$, and the very recent slight improvement of Denaux \cite[Theorem 6.2.11.]{Denaux}  together with  the previously mentioned \cite{Davydov09}  gives $$s_{q^{N}}(N, N-1)\leq  \left(\frac{N(N+1)}{2}-2\right)q-\binom{N}{2}+ \min\{7, 2q\}.$$
	
	It is easy to see that the larger $\varrho$ is, the larger is the gap between the lower and upper bounds of Theorem \ref{Den}. Our main result in this section is Corollary \ref{vegsokor} in which we get an upper bound close to the lower bound even if $\varrho$ is large.
	
	Following the proof of Theorem \ref{main1} on an upper bound of cutting blocking sets, one can get a general result for $t$-fold strong blocking sets as well.  The idea is analogous to that in Theorem \ref{main1}: we construct a $t$-fold strong blocking set in $\PG(N,q)$ as the union of a small number of randomly chosen $(N-t+1)$-dimensional subspaces. Note that a set of subspaces of $\PG(N,q)$ of dimension $N-t+1$ whose union is a $t$-fold strong blocking set is also called a set of \emph{higgledy-piggledy $(N-t+1)$-spaces} \cite{FancsaliSziklai2}. Similarly as in \cite{FancsaliSziklai1} for the case of higgledy-piggledy line sets (that is, $t=N$), Fancsali and Sziklai construct a set of higgledy-piggledy $(N-t+1)$-spaces in $\PG(N,q)$ of size $(N-t+2)(t-1)+1$, whenever $q>N+1$ \cite[Subsection~3.4]{FancsaliSziklai2}. Our random construction reaches this size only asymptotically in $q$, but it does not require $q$ to be large.

	\begin{theorem}
		There is a strong $t$-fold blocking set $\B$ in $\PG(N,q)$ consisting of the points of $m$ subspaces of dimension $N-t+1$ for 
		\[m=\lceil 
		(N-t+2)(t-1)c_1(q)+c_2(q)\rceil,\] 
		where the constants $c_1(q)$ and $c_2(q)$ are defined as
		\begin{equation} c_1(q)= \left\{
			\begin{matrix}
				\frac{-\ln q}{\ln(1 - {e^{-\frac{1}{q-2}}})} & \normalsize\mbox{for } q>2, \\
				& \\
				\frac{-\ln 2}{\ln(1 - 0.5{e^{-\frac{2}{3}}})} & \mbox{for } q=2.
			\end{matrix}
			\right. 
			\mbox{ \ \ and \ \ }
			c_2(q)= \left\{
			\begin{matrix}
				\frac{-1}{(q-2)\ln(1 - {e^{-\frac{1}{q-2}}})} & \normalsize\mbox{for } q>2, \\
				& \\
				\frac{-\ln(2e^{2/3})}{\ln(1 - 0.5{e^{-\frac{2}{3}}})} & \mbox{for } q=2.
			\end{matrix}
			\right. 
		\end{equation}
	\end{theorem}
	
	In other words, the theorem above ensures the existence of a set of higgledy-piggledy $(N-t+1)$-spaces of size $m$. Note that $c_1(q)\rightarrow 1$ as $q$ tends to infinity and $c_1(2)\approx 2.34$, whereas $c_2(q)\rightarrow 0$ as $q$ tends to infinity and $c_2(2)\approx 4.58$, $c_2(3)\approx 2.18$, and $c_2(q)<1$ for $q\geq 4$.
	
	\begin{proof}
		The proof is again an application of the first moment method. Let us choose a (multi)set of $m$ subspaces $\{H_1,\ldots, H_m\}$ of dimension $N-t+1$ in $\PG(N,q)$ uniform randomly, and let $\B=\cup_{i=1}^m H_i$. First consider the following simple observation.
		\begin{equation}\label{eq:probs2}
			\PP(\B \mbox{ is a $t$-fold strong blocking set}) \geq 
			1-\PP(\,\exists\,\Lambda \colon \dim \Lambda={t-2}, \ \forall i\:H_i\cap \Lambda\neq \emptyset )
		\end{equation} 
		
		Indeed, if there does not exist such a subspace, then the intersection with every $t-1$ dimensional subspace  $\Lambda$ must be a point set that cannot be covered by a single $t-2$ dimensional subspace, hence the intersection spans $\Lambda$ itself. 
		Here we may apply a rough estimate on the probability 
		\[
		\PP(\,\exists\,\Lambda \colon \dim \Lambda={t-2}, \ \forall i\:H_i\cap \Lambda\neq \emptyset)
		\]
		by applying the following lemma.
		
		\begin{lemma}\label{nemures2}
			Let $H$ be a subspace of dimension $N-t+1$ chosen uniform randomly, and let $\Lambda$ be a fixed $(t-2)$-dimensional subspace.
			Then 
			\begin{equation}
				\begin{split}
					\PP(H \cap \Lambda\neq \emptyset)
					< \left\{
					\begin{matrix}
						1 - \frac{1}{e^{\frac{1}{q-2}}} & \normalsize\mbox{for } q>2, \\
						& \\
						1 - \frac{1}{2e^{\frac{2}{3}}} & \mbox{for } q=2.
					\end{matrix}
					\right.
				\end{split}
			\end{equation} 
		\end{lemma}
		
		\begin{proof}
			Let us consider $\PG(N,q)$ as an $(N+1)$-dimensional vector space over $\GF(q)$, and let us count the number $A$ of $(N-t+2)$-dimensional subspaces that are disjoint from the given $(t-1)$-dimensional subspace $\Lambda$. We do this via counting the suitable bases for the $(N-t+2)$-dimensional subspace:
			\[
			A=\frac{\prod\limits_{i=0}^{N-t+1}(q^{N+1}-q^{t-1+i})}{\prod\limits_{i=0}^{N-t+1}(q^{N-t+2}-q^{i})}=\frac{\prod\limits_{i=0}^{N-t+1}q^{t-1+i}}{\prod\limits_{i=0}^{N-t+1}q^{i}}=q^{(t-1)(N-t+2)}.
			\]
			Hence  by taking into consideration the upper bound of Lemma 2.2 on Gaussian  binomials, the probability that an $(N-t+2)$-dimensional subspace intersects $\Lambda$ non-trivially is
			\[   1 - \frac{q^{(t-1)(N-t+2)}}{\qbinom{N+1}{N-t+2}} <
			\left\{
			\begin{matrix}
				1 - \frac{1}{e^{\frac{1}{q-2}}} & \normalsize\mbox{for } q>2, \\
				& \\
				1 - \frac{1}{2e^{\frac{2}{3}}} & \mbox{for } q=2.
			\end{matrix}
			\right.\]
		\end{proof}

		Clearly,
		\[
		\PP(\,\exists\,\Lambda \colon \dim \Lambda={t-2}, \ \forall i\:H_i\cap \Lambda\neq \emptyset) \leq 
		\begin{bmatrix}N+1\\ t-1\end{bmatrix}_q \left(\PP(H_1\cap \Lambda\neq \emptyset)\right)^m.
		\]
		From Lemma \ref{GBinom-becsles} on the Gaussian binomials and Lemma \ref{nemures2} we obtain
		%
		%
		%
		\[
		\PP(\,\exists\,\Lambda \colon \dim \Lambda={t-2}, \ \forall i\:H_i\cap \Lambda\neq \emptyset) \leq  
		q^{(N-t+2)(t-1)}\cdot e^{1/(q-2)}\left(1 - {e^{-\frac{1}{q-2}}}\right)^m
		\]
		for $q>2$, and 
		\[
		\PP(\,\exists\,\Lambda \colon \dim \Lambda={t-2}, \ \forall i\:H_i\cap \Lambda\neq \emptyset) \leq  
		q^{(N-t+2)(t-1)}\cdot 2e^{2/3}\left(1 - \frac{1}{2e^{\frac{2}{3}}}\right)^m
		\] 
		for $q=2$.
		
		It is easy to check that if we choose $m$ as claimed, 
		then the probability in view is strictly smaller than $1$, completing the proof.
	\end{proof}
	
	This is turn provides a bound on $\varrho$-saturating sets via Theorem \ref{kapcs}.
	
	\begin{cor}\label{vegsokor}  $$s_{q^{\varrho+1}}(N, \varrho)\leq\left\lceil c_1(q)(N-\varrho+1)\varrho+c_2(q)\right\rceil \frac{q^{N-\varrho+1}-1}{q-1}.$$
	\end{cor}
	
	Note that this improves the bound of Denaux \cite{Denaux} if $q$ and $\varrho$ is large enough; more precisely for every $\varrho>\frac{2}{3}N$ if $q$ is large enough.
	
	Let us also note that the aforementioned construction of Fancsali and Sziklai \cite{FancsaliSziklai2} yields
	\[
	s_{q^{\varrho+1}}(N, \varrho)\leq( (N-\varrho+1)\varrho+1) \frac{q^{N-\varrho+1}-1}{q-1}
	\]
	whenever $q>N+1$.
	Finally, let us mention that higgledy-piggledy lines and subspaces are also related to \emph{uniform subspace designs}. For the definition and coding theoretic applications of subspace designs, we refer to the works of Guruswami and Kopparty \cite{GK}, Guruswami, Resch, and Xing \cite{GRX} and the references therein, whereas the relation of uniform subspace designs and higgledy-piggledy subspaces can be found in the work of Fancsali and Sziklai \cite{FancsaliSziklai2}.
	
	{\bf Acknowledgement.} The authors are grateful to Lins Denaux for several remarks on the first version of the manuscript.

	\section{Appendix}
	
	{\bf Lemma 2.2.}
	
	$ \begin{bmatrix}n\\ k\end{bmatrix}_q< q^{(n-k)k}\cdot e^{1/(q-2)}$ for $q>2$  and $k>0$, and 
	
	$ \begin{bmatrix}n\\ k\end{bmatrix}_2<2^{(n-k)k+1}\cdot e^{2/3}$ for $q=2$ and $k>0$.
	
	$ \begin{bmatrix}n\\ n-2\end{bmatrix}_q < q^{2(n-2)}\cdot \frac{q}{q-1}\frac{q^2}{q^2-1}$ for $q\geq2$.

	\begin{proof} 
		\begin{equation}\label{l22eq}
			\qbinom{n}{k}=\frac{(q^n-1)(q^{n-1}-1)\cdots(q^{n-k+1}-1)} {(q^k-1)\cdots(q^2-1)(q-1)}< q^{(n-k)k}\cdot \prod_{t=1}^{k} \frac{q^t}{q^t-1}.
		\end{equation}
		This  gives the third statement as $\qbinom{n}{n-2}=\qbinom{n}{2}$. Suppose now $q>2$. Recall that $\left(1+\frac1i\right)^i$ is strictly increasing, whence $\left(1+\frac{1}{(q-1)^t}\right) \leq \left(1+\frac{1}{(q-1)^k}\right)^{(q-1)^{k-t}}$ follows for $t\leq k$. Thus we have
		\[ 
		\begin{bmatrix}n\\ k\end{bmatrix}_q< 
		q^{(n-k)k}\cdot \prod_{t=1}^{k} \frac{q^t}{q^t-1}
		\leq q^{(n-k)k}\prod_{t=1}^{k} \left(1+\frac{1}{(q-1)^t}\right)\leq
		q^{(n-k)k}\left(1+\frac{1}{(q-1)^k}\right)^{\sum_{t=1}^k(q-1)^{k-t}}
		\]
		\[= q^{(n-k)k}\left(1+\frac{1}{(q-1)^k}\right)^{\frac{(q-1)^{k}-1}{q-2}}<q^{(n-k)k}\cdot e^{1/(q-2)}.
		\]
		For $q=2$, \eqref{l22eq} gives
		\[
		\begin{bmatrix}n\\ k\end{bmatrix}_2 < 
		2^{(n-k)k}\cdot\prod_{i=1}^k \left(1+\frac{1}{2^t-1}\right)
		\leq 2^{(n-k)k}\cdot 2\cdot \prod_{t=0}^{k-2} \left(1+\frac{1}{3\cdot2^t}\right)<2^{(n-k)k+1}\cdot  \left(1+\frac{1}{3\cdot2^{k-2}}\right)^{2^{k-1}},\]
		which implies 
		$ \begin{bmatrix}n\\ k\end{bmatrix}_2<2^{(n-k)k+1}\cdot e^{2/3}$.
	\end{proof}

	Here comes the precise deduction of Inequality \eqref{nemures}. We formulate a lemma first.
	
	\begin{lemma}\mbox{}
		
		\begin{itemize}
			\item If $0\leq a<b$, then $\frac{\theta_a+\frac1q}{\theta_b} \leq q^{a-b} < \frac{\theta_a+1}{\theta_b}$.
			\item If $0\leq a\leq b-2$, then $\frac{\theta_a+\theta_b}{q} > 2\theta_a + 1$.
			\item $\qbinom{k+1}{2}=\frac{\theta_k\theta_{k-1}}{q+1}$.
		\end{itemize}
	\end{lemma}
	\begin{proof}
		The first assertion follows from $\theta_b=\theta_a(q^{b-a})+\theta_{b-a-1}$ and $q^{b-a-1}\leq \theta_{b-a-1} < q^{b-a}$.
		As for the second, $b\geq a+2$ yields $\theta_b\geq q^2\theta_a+q+1$, whence $\theta_a + \theta_b \geq (q^2+1)\theta_a + q + 1 > 2q\theta_a + q$, as asserted.
		$\qbinom{k+1}{2}$ is the number of lines in $\PG(k,q)$, which space has $\theta_k$ points, each incident with $\theta_{k-1}$ lines, all of which have $q+1$ points.
	\end{proof}
	
	\noindent\textit{Proof of Inequality \eqref{nemures}}.
	Recall that $d\leq N-2$ is the dimension of a fixed subspace $\Lambda$, and $\ell_1$ is a line chosen uniform randomly.
	\[
	\PP(\ell_1 \cap \Lambda \neq \emptyset)= \frac{ \qbinom{d+1}{2}   + \qbinom{d+1}{1}  \cdot \frac{1}{q} \left( \qbinom{N+1}{1} -\qbinom{d+1}{1} \right)}{\qbinom{N+1}{2}}=
	\frac{ \frac{\theta_d\theta_{d-1}}{q+1} +\frac{\theta_d}{q}\left(\theta_N-\theta_d\right) } {\frac{\theta_{N}\theta_{N-1}}{q+1}}=
	\]
	\[
	=\frac{\theta_d\theta_{d-1}  + \left(1+\frac{1}{q}\right)\theta_d(\theta_N-\theta_d)}{\theta_N\theta_{N-1}}=
	\frac{\theta_d\theta_{d-1}  + \left(\theta_d+\theta_{d-1}+\frac{1}{q}\right)(\theta_N-\theta_d)}{\theta_N\theta_{N-1}}=
	\]
	\[
	=\frac{\theta_d+\theta_{d-1}+\frac1q}{\theta_{N-1}} - \frac{\theta_d^2+\frac{\theta_d}{q}}{\theta_{N}\theta_{N-1}}
	=\frac{\theta_d+\frac1q + \theta_{d-1} +\frac1q}{\theta_{N-1}} - \frac{\theta_d^2+\frac{\theta_d}{q} +\frac{\theta_N}{q}}{\theta_{N}\theta_{N-1}}<
	\]
	\[
	<\frac{\theta_d+\frac1q + \theta_{d-1} +\frac1q}{\theta_{N-1}} - \frac{\theta_d^2+2\theta_d +1}{\theta_{N}\theta_{N-1}} = \frac{\theta_d+\frac1q}{\theta_{N-1}} + \frac{\theta_{d-1} +\frac1q}{\theta_{N-1}} - \frac{\theta_d+1}{\theta_N}\cdot\frac{\theta_d+1}{\theta_{N-1}} <
	\]
	\[
	< q^{d-N+1}+q^{d-N}-q^{2d-2N+1}.
	\]
	\hfill\qed

\end{document}